\documentclass[12pt]{article}
\def\date{\today}
\usepackage{amsmath, amssymb, amsfonts, amsthm, color, enumitem, graphicx}
\usepackage[margin=1in]{geometry}
\usepackage[bookmarks=false]{hyperref}

\newtheorem{theorem}{Theorem}
\newtheorem{lemma}[theorem]{Lemma}
\newtheorem{corollary}[theorem]{Corollary}
\newtheorem{proposition}[theorem]{Proposition}

\newtheorem{claim}{Claim}
\newtheorem*{claim*}{Claim}

\theoremstyle{definition}
\newtheorem{remark}[theorem]{Remark}

\newtheorem*{definition*}{Definition}
\newtheorem*{remark*}{Remark}
\newtheorem*{remarks*}{Remarks}

\newcommand{\abs}[1]{\left|{#1}\right|}

\begin{document}

\phantom{a}\vskip .25in
\centerline{{\large \bf  FOUR EDGE-INDEPENDENT SPANNING TREES}
\footnote{%
Partially supported by NSF under Grant No.~DMS-1202640.}
\footnote{Presented at the 29th Cumberland Conference on Combinatorics, Graph Theory and Computing at Vanderbilt University.}
}

\vskip.4in
\centerline{{\bf Alexander Hoyer}}
\medskip
\centerline{and}
\medskip
\centerline{{\bf Robin Thomas}}
\medskip
\centerline{School of Mathematics}
\centerline{Georgia Institute of Technology}
\centerline{Atlanta, Georgia  30332-0160, USA}

\vskip 1in \centerline{\bf ABSTRACT}
\bigskip\bigskip

{
\noindent
We prove an ear-decomposition theorem for $4$-edge-connected graphs and use it to prove that for every $4$-edge-connected graph $G$ and every $r\in V(G)$, there is a set of four spanning trees of $G$ with the following property. For every vertex in $G$, the unique paths back to $r$ in each tree are edge-disjoint. Our proof implies a polynomial-time algorithm for constructing the trees.
}

\vfill \baselineskip 11pt \noindent March 2017. Revised October 2017.

\vfil\eject
\baselineskip 18pt

\vskip .75in

\section{Introduction}
\noindent
If $r$ is a vertex of a graph $G$, two subtrees $T_1,T_2$ of $G$ are \textit{edge-independent with root $r$} if each tree contains $r$, and for each $v\in V(T_1)\cap V(T_2)$, the unique path in $T_1$ between $r$ and $v$ is edge-disjoint from the unique path in $T_2$ between $r$ and $v$. Larger sets of trees are called \textit{edge-independent with root $r$} if they are pairwise edge-independent with root $r$.

Itai and Rodeh~\cite{ItaiRodeh} posed the Edge-Independent Tree Conjecture, that for every $k$-edge-connected graph $G$ and every $r\in V(G)$, there is a set of $k$ edge-independent spanning trees of $G$ rooted at $r$. Here, we prove the case $k=4$ of the Edge-Independent Tree Conjecture. That is, we prove the following:

\begin{theorem}\label{thmTrees}
If $G$ is a $4$-edge-connected graph and $r\in V(G)$, then there exists a set of four edge-independent spanning trees of $G$ rooted at $r$.
\end{theorem}

There is a similar conjecture which has been studied in parallel, concerning vertices rather than edges. If $r$ is a vertex of $G$, two subtrees $T_1,T_2$ of $G$ are \textit{independent with root $r$} if each tree contains $r$, and for each $v\in V(T_1)\cap V(T_2)$, the unique path in $T_1$ between $r$ and $v$ is internally vertex-disjoint from the unique path in $T_2$ between $r$ and $v$. Larger sets of trees are called \textit{independent with root $r$} if they are pairwise independent with root $r$.

Itai and Rodeh~\cite{ItaiRodeh} also posed the Independent Tree Conjecture, that for every $k$-connected graph $G$ and for every $r\in V(G)$, there is a set of $k$ independent spanning trees of $G$ rooted at $r$.

The case $k=2$ of each conjecture was proven by Itai and Rodeh~\cite{ItaiRodeh}. The case $k=3$ of the Independent Tree Conjecture was proven by Cheriyan and Maheshwari~\cite{CheriMahesh}, and then independently by Zehavi and Itai~\cite{ZehaviItai}. Huck~\cite{Huck} proved the Independent Tree Conjecture for planar graphs (with any $k$). Building on this work and that of Kawarabayashi, Lee, and Yu~\cite{klylovasz}, the case $k=4$ of the Independent Tree Conjecture was proven by Curran, Lee, and Yu across two papers~\cite{clychain,cly}. The Independent Tree Conjecture is open for nonplanar graphs with $k>4$.

In 1992, Khuller and Schieber~\cite{Khuller} published a later-disproven argument that the Independent Tree Conjecture implies the Edge-Independent Tree Conjecture. Gopalan and Ramasubramanian~\cite{Gopalan} demonstrated that Khuller and Schieber's proof fails, but salvaged the technique, and proved the case $k=3$ of the Edge-Independent Tree Conjecture by reducing it to the case $k=3$ of the Independent Tree Conjecture. Schlipf and Schmidt~\cite{Schlipf} provided an alternate proof of the case $k=3$ of the Edge-Independent Tree Conjecture, which does not rely on the Independent Tree Conjecture. The case $k=4$ of the Edge-Independent Tree Conjecture is proven here, while the case $k>4$ remains open.

By adapting the technique of Schlipf and Schmidt~\cite{Schlipf}, we prove an edge analog of the planar chain decomposition of Curran, Lee, and Yu~\cite{clychain}. We then use this decomposition to create two edge numberings which define the required trees.

The conjectures are related to network communication with redundancy. If $G$ represents a communication network, one can wonder if information can be broadcast through the entire network with resistance to edge failures (i.e. it would require $k$ simultaneous edge failures to disconnect a client from every broadcast). The Edge-Independent Tree Conjecture implies that the absence of edge bottlenecks of size less than $k$ is necessary and sufficient for a redundant broadcast to be possible from any source $r$. The Independent Tree Conjecture answers the analogous problem where vertex failures are the concern, rather than edge failures.

\section{The Chain Decomposition}
\noindent
In this paper, a \textit{graph} will refer to what is commonly called a multigraph. That is, there may be multiple edges between the same pair of vertices (``parallel edges") and an edge may connect a vertex to itself (a ``loop"). All paths and cycles are simple, meaning they have no repeated vertices or edges. We consider a loop to induce a cycle of length one and a pair of parallel edges to induce a cycle of length two. Also, the presence of a loop increases the degree of a vertex by two. We will use the overline notation $\overline{H}$ to name specific subgraphs, rather than for the graph complement.

Throughout this section, fix a graph $G$ with $\abs{V(G)}\geq1$ and a vertex $r\in V(G)$. We begin by defining a decomposition analogous to the planar chain decomposition in~\cite{clychain}.

\begin{definition*}
An \textit{up chain} of $G$ with respect to a pair of edge-disjoint subgraphs ($H$, $\overline{H}$) is a subgraph of $G$, edge-disjoint from $H$ and $\overline{H}$, which is either:
\begin{enumerate}[label=\roman*]
\item A path with at least one edge such that every vertex is either $r$ or has degree at least two in $\overline{H}$, and the ends are either $r$ or are in $H$, OR
\item A cycle such that every vertex is either $r$ or has degree at least two in $\overline{H}$, and some vertex $v$ is either $r$ or has degree at least two in $H$. We will consider $v$ to be both ends of the chain, and all other vertices in the chain to be internal vertices.
\end{enumerate}
Chains which are paths will be called \textit{open} and chains which are cycles will be called \textit{closed}, analogous to the standard ear decomposition.
\end{definition*}

\begin{definition*}
A \textit{down chain} of $G$ with respect to a pair of edge-disjoint subgraphs ($H$, $\overline{H}$) is an up chain with respect to ($\overline{H}$, $H$).
\end{definition*}

\begin{definition*}
A \textit{one-way chain} of $G$ with respect to the pair of edge-disjoint subgraphs ($H$, $\overline{H}$) is a subgraph of $G$, induced by an edge $e\notin H\cup\overline{H}$ with ends $u$ and $v$, such that $u$ is either $r$ or has degree at least two in $H$, and $v$ is either $r$ or has degree at least two in $\overline{H}$. We call $u$ the \textit{tail} of the chain and $v$ the \textit{head}.
\end{definition*}

\begin{definition*}
Let $G_0,G_1,\ldots,G_m$ be a sequence of subgraphs of $G$. Denote $H_i=G_0\cup G_1\cup\cdots\cup G_{i-1}$ and $\overline{H_i}=G_{i+1}\cup G_{i+2}\cup\cdots\cup G_m$, so that $H_0$ and $\overline{H_m}$ are the null graph. We say that the sequence $G_0,G_1,\ldots,G_m$ is a \textit{chain decomposition} of $G$ rooted at $r$ if:
\begin{enumerate}
\item The sets $E(G_0),E(G_1),\ldots,E(G_m)$ partition $E(G)$, AND
\item For $i=0,\ldots,m$, the subgraph $G_i$ is either an up chain, a down chain, or a one-way chain with respect to the subgraphs $(H_i,\overline{H_i})$.
\end{enumerate}
\end{definition*}

\begin{figure}[h]
\centering
\includegraphics[scale=0.6]{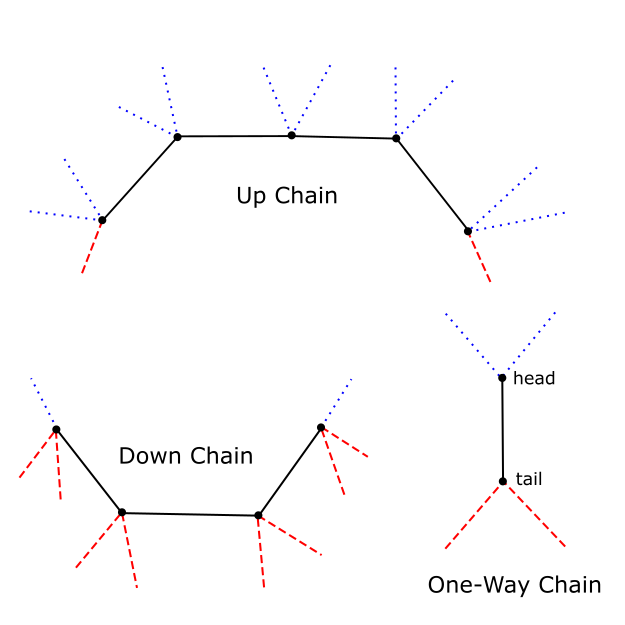}
\caption{An illustration of an up chain of length 4, a down chain of length 3, and a one-way chain. The red/dashed edges are in earlier chains, while the blue/dotted edges are in later chains.}
\label{figChains}
\end{figure}

\begin{definition*}
The \textit{chain index} of $e\in E(G)$, denoted $CI(e)$, is the index of the chain containing $e$.
\end{definition*}

\begin{definition*}
An up chain $G_i$ is \textbf{minimal} if no internal vertex of $G_i$ is in $\{r\}\cup V\left(H_i\right)$.
\end{definition*}

\begin{definition*}
A down chain $G_i$ is \textbf{minimal} if no internal vertex of $G_i$ is in $\{r\}\cup V\left(\overline{H_i}\right)$.
\end{definition*}

\begin{definition*}
A chain decomposition is \textbf{minimal} if all of its up chains and down chains are minimal.
\end{definition*}

\begin{remarks*}
\ \\\vspace{-15pt}
\begin{enumerate}
\item A minimal up chain is a special case of an ear in the standard ear decomposition.
\item The chain decomposition is symmetric in the following sense. If $G_0,G_1,\ldots,G_m$ is a chain decomposition rooted at $r$, then $G_m,G_{m-1},\ldots,G_0$ is a chain decomposition rooted at $r$, with the up and down chains switched and the heads and tails of one-way chains switched. Throughout this paper, we will refer to this fact as ``symmetry".
\item $G_0$ is either a closed up chain ending at $r$ or a one-way chain with $r$ as the tail, and $G_m$ is either a closed down chain ending at $r$ or a one-way chain with $r$ as the head.
\item In the planar chain decomposition in~\cite{clychain}, up chains and down chains are analogous to the corresponding open chains. The elementary chain is analogous to a one-way chain.
\end{enumerate}
\end{remarks*}

\begin{remark}\label{rmkMin}
An up chain or down chain may be subdivided into several minimal chains by breaking at the offending internal vertices. These minimal chains may then be inserted consecutively to the decomposition at the index of the old chain. In this way, one can easily obtain a minimal chain decomposition from any chain decomposition.
\end{remark}

We will prove Theorem~\ref{thmTrees} by combining the following results:

\begin{theorem}\label{thmChains}
If $G$ is a $4$-edge-connected graph and $r\in V(G)$, then $G$ has a chain decomposition rooted at $r$.
\end{theorem}

\begin{theorem}\label{thmTreesFromChains}
Suppose $G$ is a graph with no isolated vertices. If $G$ has a chain decomposition rooted at some $r\in V(G)$, then there exists a set of four edge-independent spanning trees of $G$ rooted at $r$.
\end{theorem}

\section{Preliminary Results}

While not needed for our main results, the following proposition demonstrates how the chain decomposition fits in with the various decompositions used in other cases of the Independent Tree Conjecture and Edge-Independent Tree Conjecture. A partial chain decomposition and its complement are ``almost $2$-edge-connected" in the following sense.

\begin{proposition}\label{prp2con}
Suppose $G_0,G_1,\ldots,G_m$ is a chain decomposition of a graph $G$ rooted at $r$. Then for $i=1,\ldots,m$, $H_i$ and $\overline{H_{i-1}}$ are connected. Further, if $e$ is a cut edge of $H_i$ (resp. $\overline{H_{i-1}}$), then $e$ induces a one-way chain and one component of $H_i-e$ (resp. $\overline{H_{i-1}}-e$) contains one vertex and no edges.
\end{proposition}

\begin{proof}
By symmetry, we need only prove the result for the $H_i$'s. The connectivity follows from the fact that every type of chain is connected and contains at least one vertex in an earlier chain.

Suppose $e$ is a cut edge of some $H_i$. Since $e$ is an edge in $H_i$, we have $CI(e)<i$ and $H_{CI(e)}\subset H_i$. We also know that $H_{CI(e)}$ is connected by the previous paragraph. Then $e$ cannot be part of an up chain, or else $e$ would be part of a cycle formed by the chain $G_{CI(e)}$ and a path in $H_{CI(e)}$ between the ends of $G_{CI(e)}$ (if $G_{CI(e)}$ is open; else the chain itself is a cycle). Also, $e$ cannot be part of a down chain, or else $e$ would be part of a cycle formed by $e$ and a path in $H_{CI(e)}$ between the ends of $e$. Therefore, $e$ induces a one-way chain.

Let $C$ be the component of $H_i-e$ not containing $r$, and suppose for the sake of contradiction that $C$ contains an edge. Let $e'$ be an edge of $C$ with minimal chain index. Consider $G_{CI(e')}$, the chain containing $e'$. Regardless of the chain type, some vertices in $V(G_{CI(e')})$ are incident to at least two edges in $H_{CI(e')}\subset H_i$ since $r\notin C$, so one of these edges is not $e$. This contradicts the minimality of $CI(e')$.
\end{proof}

The next lemma and its corollary will allow us to ignore the possibility of loops in the graph when convenient.

\begin{lemma}\label{lmaNoLoops}
Suppose $G_0,G_1,\ldots,G_m$ is a chain decomposition of $G$ rooted at $r$. If $v\neq r$ is in $H_i$ (resp $\overline{H_i}$), then $v$ is incident to a non-loop edge in $H_i$ (resp $\overline{H_i}$). If $v$ has degree at least two in $H_i$ (resp. $\overline{H_i}$), then $v$ is incident to two distinct non-loop edges in $H_i$ (resp. $\overline{H_i}$).
\end{lemma}

\begin{proof}
Note that the second claim in the lemma implies the first, since a loop increases the degree of a vertex by 2, so it suffices to prove the second claim in the lemma.

Suppose $v$ is incident to a loop, which by symmetry we may assume is in $H_i$. Of all loops incident to $v$, choose the one with minimal chain index $j<i$. Consider the chain classification of $G_j$. The chain definitions all coincide for a loop, and require that $v$($\neq r$) has degree at least two in $H_j$. By the minimality of $j$, $v$ is not incident to any loops in $H_j$. It follows that $v$ is incident to two distinct non-loop edges in $H_j\subset H_i$.
\end{proof}

\begin{corollary}\label{corNoLoops}
Suppose $G_0,G_1,\ldots,G_m$ is a chain decomposition of $G$ rooted at $r$, and $e\in E(G_i)$ is a loop. Then $G_0,G_1,\ldots,G_{i-1},G_{i+1},\ldots,G_m$ is a chain decomposition of $G-e$ rooted at $r$. Further, if $G$ has no isolated vertices, then $G-e$ has no isolated vertices.
\end{corollary}

\begin{proof}
The first claim follows from the preceding lemma. For the second, observe that if $e$ is the only edge incident to its end, then it fails the conditions for every chain definition.
\end{proof}

Next, we prove the following useful fact about minimal chain decompositions.

\begin{lemma}\label{lmaDeg2}
Suppose $G$ is a graph with no isolated vertices, $G_0,G_1,\ldots,G_m$ is a minimal chain decomposition of $G$ rooted at $r$, and $v\in V(G)$ with $v\neq r$. Then there are indices $i,j$ so that $v$ has degree exactly two in $H_i$ and $\overline{H_j}$.
\end{lemma}

\begin{proof}
By symmetry, we need only find $i$. Since $G$ has no isolated vertices, $v$ is in some chain. Consider the chain $G_{i_0}$ containing $v$ so that $i_0$ is minimal. Note that $v\notin V(H_{i_0})$.

If $G_{i_0}$ is an up chain, then $v$ is an internal vertex of $G_{i_0}$ since $v\notin V(H_{i_0})$, so $v$ has degree two in $G_{i_0}$ and degree at least two in $\overline{H_{i_0}}$. Therefore $\overline{H_{i_0}}$ is not null, so $i_0<m$. Then $i=i_0+1$ completes the proof.

The chain $G_{i_0}$ is not a down chain since $v\notin V(H_{i_0})$.

So we may assume that $G_{i_0}$ is a one-way chain, and $v$ must be the head since $v\notin V(H_{i_0})$. Therefore $v$ has degree at least two in $\overline{H_{i_0}}$, so we may consider the next chain to contain $v$, say $G_{i_1}$. Note that $v$ has degree one in $H_{i_1}$ by the definition of $i_1$.

If $G_{i_1}$ is an up chain, then it is open and $v$ is an end of the chain, since the chain decomposition is minimal and $v$ has degree one in $H_{i_1}$. The chain $G_{i_1}$ is not a down chain since $v$ has degree one in $H_{i_1}$. If $G_{i_1}$ is a one-way chain, then $v$ is the head since $v$($\neq r$) does not have degree at least two in $H_{i_1}$. In all cases, $v$ has degree one in $G_{i_1}$ and degree at least two in $\overline{H_{i_1}}$. Therefore $\overline{H_{i_1}}$ is not null, so $i_1<m$. Then $i=i_1+1$ completes the proof.
\end{proof}

Finally, we show that the chain decomposition implies a minimum degree result.

\begin{lemma}\label{lmaMinDeg4}
Suppose $G$ is a graph with no isolated vertices, $G_0,G_1,\ldots,G_m$ is a chain decomposition of $G$ rooted at $r$, and $v\in V(G)$ with $v\neq r$. Then $v$ has degree at least $4$.
\end{lemma}

\begin{proof}
By Corollary \ref{corNoLoops}, we may assume that there are no loops in $G$. If $v$ is in an up chain $G_i$, then $v$ has degree at least $2$ in $\overline{H_i}$, and either degree $2$ in $G_i$ (if $v$ is internal) or degree at least $1$ in $G_i$ and degree at least $1$ in $H_i$ (if $v$ is an end). Either way, $v$ has degree at least $4$ in $G$. By symmetry, the same is true if $v$ is in a down chain.

So we may assume that the only chains containing $v$ are one-way chains. Since $G$ has no isolated vertices, there is at least one such chain $G_j$. Then $v$ has degree $1$ in $G_j$ and degree at least $2$ in $H_j$ (if $v$ is the tail) or $\overline{H_j}$ (if $v$ is the head). We conclude that $v$ has degree at least $3$ in $G$.

Assume for the sake of contradiction that $v$ does not have degree at least $4$. Then $v$ has degree $3$ and is in exactly three one-way chains, say $G_{\ell_1}$, $G_{\ell_2}$, $G_{\ell_3}$ with $\ell_1<\ell_2<\ell_3$. Consider $G_{\ell_2}$. Since we know all of the chains containing $v$, we can say that $v$ has degree $1$ in $H_{\ell_2}$ and degree $1$ in $\overline{H_{\ell_2}}$. This contradicts the definition of a one-way chain, as $v$ can be neither the head nor the tail of the chain $G_{\ell_2}$. We conclude that $v$ has degree at least $4$ as desired.
\end{proof}

\begin{remark*}
If $\abs{V(G)}\geq2$ in addition to $G$ having a chain decomposition and no isolated vertices, then $G$ is $4$-edge-connected so $r$ has degree at least $4$ as well. However, we will not need this result, and it will follow from Corollary \ref{corSummary}.
\end{remark*}

\section{The Mader Construction}

We will adapt the strategy of Schlipf and Schmidt~\cite{Schlipf} in order to construct a chain decomposition. In particular, we will use a construction method for $k$-edge-connected graphs due to Mader~\cite{Mader}. We limit our description of the construction to the needed case $k=4$, since the method is more complicated for odd $k$.

\begin{definition*} A \textit{Mader operation} is one of the following operations:
\begin{enumerate}
\item Add an edge between two (not necessarily distinct) vertices.

\item Consider two distinct edges, say $e_1$ with ends $x$, $y$ and $e_2$ with ends $z$, $w$, and ``pinch" them as follows. Delete the edges $e_1$ and $e_2$, add a new vertex $v$, then add the new edges $e_x,e_y,e_z,e_w$ with one end $v$ and the other end $x,y,z,w$ respectively. While $e_1$ and $e_2$ must be distinct, the ends $x,y,z,w$ need not be. In this case, $v$ will have parallel edges to any repeated vertex.
\end{enumerate}
\end{definition*}

\begin{theorem}[{\cite[Corollary~14]{Mader}}]\label{thmMader}
 A graph $G$ is $4$-edge-connected if and only if, for any $r\in V(G)$, one can construct $G$ in the following way. Begin with a graph $G^0$ consisting of $r$ and one other vertex of $G$, connected by four parallel edges. Then, repeatedly perform Mader operations to obtain $G$.
\end{theorem}

\begin{remark*}
Mader does not explicitly state that one can include a fixed vertex $r$ in $G^0$, but it follows from his work. His proof starts with $G$, and then reverses one of the Mader operations while maintaining $4$-edge-connectivity. An edge can be deleted unless $G$ is minimally $4$-edge-connected, in which case he finds two vertices of degree $4$ in his Lemma 13. He then shows that any degree $4$ vertex can be ``split off" (the reverse of a pinch) in his Lemma 9, so we can always split off a vertex not equal to $r$.
\end{remark*}

\section{Proof of Theorem~\ref{thmChains}}

Due to Theorem~\ref{thmMader}, it suffices to prove that a chain decomposition can be maintained through a Mader operation. The decomposition in the starting graph $G^0$ is as follows. Two of the edges form a closed up chain. The remaining two edges form a closed down chain.

Suppose the graph $G'$ is obtained from the graph $G$ by a Mader operation, with both graphs $4$-edge-connected. Assume that we have a chain decomposition $G_0,G_1,\ldots,G_m$ of $G$. By Remark~\ref{rmkMin}, we may assume that we have a minimal chain decomposition. We wish to create a new chain decomposition of $G'$.

\subsection{Adding an Edge}

Suppose $G'$ is obtained from $G$ by adding an edge with ends $u$, $v$. If one of the ends is the root $r$, we can classify the new edge as a one-way chain with tail $r$ at, say, the very beginning of the chain decomposition. The head must have at least two incident edges in later chains, since all chains are later.

If neither end is $r$, choose the minimal index $i$ such that $u$ or $v$ has degree exactly two in $H_i$, guaranteed to exist by Lemma~\ref{lmaDeg2}. Note that $i\geq1$ since $H_0$ is null. Without loss of generality, $u$ has degree exactly two in $H_i$. By the definition of $i$, $v$ has degree at most two in $H_i$, and therefore degree at least two in $\overline{H_{i-1}}$. We classify the new edge as a one-way chain with tail $u$ and head $v$, between the chains $G_{i-1}$ and $G_i$.

We consider the impact of these changes on other chains in the graph. A new chain was added, but none of the other chains changed index relative to each other. Vertices may have increased degree in the $H_i$'s or the $\overline{H_i}$'s due to the new edge, but increasing degree does not invalidate any chain types. Note that some chains may no longer be minimal, so the new chain decomposition in $G'$ is not necessarily minimal.

\subsection{Pinching Edges}

Suppose $G'$ is obtained from $G$ by pinching the edges $e_1$ with ends $x$, $y$ and $e_2$ with ends $z$, $w$, replacing them with edges $e_x,e_y,e_z,e_w$. We will use the notation $J_1=G_{CI(e_1)}=P_xe_1P_y$ for the chain containing $e_1$, where $P_x$ is the subpath between $x$ and an end of $J_1$ so that $e_1\notin E(P_x)$, and $P_y$ is defined similarly. Note that $P_x$ (resp. $P_y$) may have no edges if $x$ (resp. $y$) is an end of $J_1$. In the same way, we will use the notation $J_2=G_{CI(e_2)}=P_ze_2P_w$ for the chain containing $e_2$.

We now prove several claims to deal with all possible chain classification and chain index combinations for $J_1$ and $J_2$.


\begin{claim}\label{clm1}
If $CI(e_1)=CI(e_2)$, then $G'$ has a chain decomposition rooted at $r$.
\end{claim}

\begin{proof}
If $CI(e_1)=CI(e_2)$, then $J_1=J_2$. Without loss of generality, $e_1\in E(P_z)$ and $e_2\in E(P_y)$, so that the chain can be written as $J_1=J_2=P_xe_1(P_y\cap P_z)e_2P_w$ (where $P_y\cap P_z$ may have no edges if $y=z$). Recall that $e_1$ and $e_2$ are distinct, so $J_1=J_2$ is not a one-way chain.

By symmetry, we may assume $J_1=J_2$ is an up chain. In $G'$, we replace the chain $J_1=J_2$ with the following chains (in the listed order); see Figure~\ref{figClm1} for an illustration:

\begin{enumerate}
\item $P_xe_xe_wP_w$. This is an up chain. Since the edges $e_y$ and $e_z$ have not yet been used, the new vertex $v$ is incident to two edges in later chains.

\item $e_y$. This is a one-way chain with tail $v$ and head $y$. The tail $v$ is incident to two edges in earlier chains, namely $e_x$ and $e_w$. The head $y$ is incident to two edges in later chains since it was an internal vertex in the old up chain $J_1=J_2$.

\item $e_z$. This is a one-way chain with tail $v$ and head $z$. The tail $v$ is incident to two edges in earlier chains, namely $e_x$ and $e_w$. The head $z$ is incident to two edges in later chains since it was an internal vertex in the old up chain $J_1=J_2$.

\item $(P_y\cap P_z)$. Only add this chain if $P_y\cap P_z$ contains an edge. This is an up chain. The new ends $y,z$ are each incident to an edge in an earlier chain ($e_y$ and $e_z$, respectively) and are each incident to two edges in later chains since they were interior vertices of the old up chain $J_1=J_2$.
\end{enumerate}

\begin{figure}[h]
\centering
\includegraphics[scale=0.6]{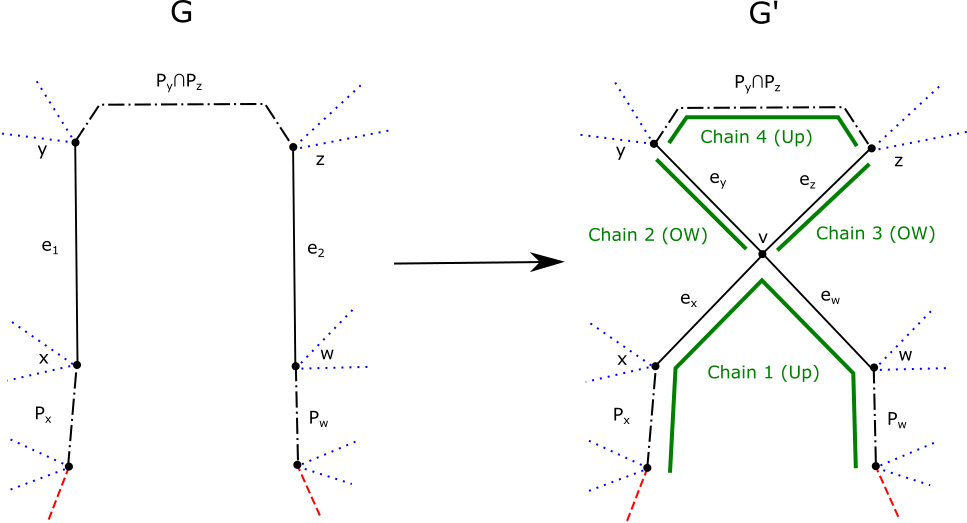}
\caption{An illustration of the procedure in Claim 1. The original up chain $J_1=J_2$ is on the left, while its replacements in $G'$ are on the right. The red/dashed edges are in earlier chains than $J_1=J_2$, while the blue/dotted edges are in later chains than $J_1=J_2$. The black/dashed-and-dotted segments represent paths which may have any length (including 0).}
\label{figClm1}
\end{figure}

We consider the impact of these replacements on other chains in the graph. We inserted most of the edges of the old chain $J_1=J_2$ at the same chain index $CI(e_1)=CI(e_2)$, preventing any changes. The exception is the pinched edges $e_1$ and $e_2$ which were deleted, but the ends each received new incident edges $e_x,e_y,e_z,e_w$ inserted at the same chain index $CI(e_1)=CI(e_2)$. Thus, we have maintained the chain decomposition. This proves Claim \ref{clm1}.
\end{proof}

Without loss of generality, we assume the following for the remainder of the proof:
\begin{itemize}
\item $CI(e_1)<CI(e_2)$.
\item If $J_1$ is a one-way chain, then $x$ is the tail and $y$ the head.
\item If $J_2$ is a one-way chain, then $z$ is the tail and $w$ the head.
\end{itemize}


\begin{claim}\label{clm2}
Suppose that either $J_1$ is a one-way chain whose head $y$ has degree one in $H_{CI(e_2)}$, or $J_2$ is a one-way chain whose tail $z$ has degree one in $\overline{H_{CI(e_1)}}$. Then $G'$ has a chain decomposition rooted at $r$.
\end{claim}

\begin{proof}
By symmetry, we may assume $J_1$ is a one-way chain whose head $y$ has degree one in $H_{CI(e_2)}$.

First, we replace $J_1$ with $e_x$. This is a one-way chain with tail $x$ and head $v$. The tail $x$ was the tail of the old one-way chain $J_1$. The head $v$ has two (in fact three) incident edges in later chains, namely $e_y$, $e_z$, $e_w$.

\begin{itemize}
\item Case 1: $J_2$ is an up chain. Since $y$ has degree one in $H_{CI(e_2)}$, if $J_2$ is closed then $y$ is not the end of $J_2$. By swapping $z$ and $w$ if necessary, we may assume that $y$ is not the end of $J_2$ in $P_z$. Thus, the end of $J_2$ in $P_z$ is still either $r$ or incident to an edge in an earlier chain, despite having not placed $e_y$ yet. We use the edges of $J_2$ and $e_y$, $e_z$, $e_w$ to construct chains at the index $CI(e_2)$ as follows:

\begin{enumerate}
\item $P_ze_z$. This is an up chain. The new end, $v$, has one incident edge in an earlier chain ($e_x$) and two incident edges in later chains ($e_y$, $e_w$). By assumption, the old end in $P_z$ is still either $r$ or incident to an edge in an earlier chain.

\item $e_y$. This is a one-way chain with tail $v$ and head $y$. The tail $v$ is incident two edges in earlier chains ($e_x$, $e_z$). The head $y$ is either $r$ or incident to two edges in later chains, since $y$ has degree one in $H_{CI(e_2)}$ by assumption.

\item $e_w$. This is a one-way chain with tail $v$ and head $w$. The tail $v$ has two (in fact three) incident edges in earlier chains ($e_x$, $e_y$, $e_z$). The head $w$ is either $r$ or incident to two edges in later chains, since it was part of the old up chain $J_2$.

\item $P_w$. Only add this if $P_w$ contains an edge. This is an up chain. The new end, $w$, has one incident edge in an earlier chain ($e_w$) and two incident edges in later chains since it was an internal vertex of the old up chain $J_2$. Since we placed $e_y$ above, the end of $J_2$ in $P_w$ has is either $r$ or incident to an end in an earlier chain, even if the end is $y$.
\end{enumerate}

\item Case 2: $J_2$ is a down chain. Since $y$ has degree one in $H_{CI(e_2)}$, $y\notin V(J_2)$, so each vertex of $J_2$ is still either $r$ or incident to two edges in earlier chains, despite having not placed $e_y$ yet. We use the edges of $J_2$ and $e_y$, $e_z$, $e_w$ to construct chains at the index $CI(e_2)$ as follows:

\begin{enumerate}
\item $P_w$. Only add this if $P_w$ contains an edge. This is a down chain. The new end, $w$, has one incident edge in a later chain ($e_w$) and two incident edges in earlier chains since it was an internal vertex of the old down chain $J_1$.

\item $e_w$. This is a one-way chain with tail $w$ and head $v$. The tail $w$ is either $r$ or incident to two edges in earlier chains since it was part of the old down chain $J_2$. The head $v$ is incident to two edges in later chains ($e_y$, $e_z$).

\item $P_ze_z$. This is a down chain. The new end, $v$, has one incident edge in a later chain ($e_y$) and two incident edges in earlier chains ($e_x$, $e_w$).

\item $e_y$. This is a one-way chain with tail $v$ and head $y$. The tail $v$ has two (in fact three) incident edges in earlier chains ($e_x$, $e_z$, $e_w$). The head $y$ is either $r$ or incident to two edges in later chains since $y$ has degree one in $H_{CI(e_2)}$ and $y\notin V(J_2)$ by assumption, so $y$ has degree at least three in $\overline{H_{CI(e_2)}}$ unless it is $r$.
\end{enumerate}

\item Case 3: $J_2$ is a one-way chain. Since $y$ has degree one in $H_{CI(e_2)}$, $y\neq z$ so the tail $z$ is still either $r$ or incident to two edges in earlier chains, despite having not placed $e_y$ yet. We use the edges $e_y$, $e_z$, $e_w$ to construct chains at the index $CI(e_2)$ as follows:

\begin{enumerate}
\item $e_z$. This is a one-way chain with tail $z$ and head $v$. The tail $z$ is either $r$ or incident to two edges in earlier chains as discussed above. The head $v$ is incident to two edges in later chains ($e_y$, $e_w$).

\item $e_w$. This is a one-way chain with tail $v$ and head $w$. The tail $v$ is incident to two edges in earlier chains ($e_x$, $e_z$). The head $w$ is either $r$ or incident to two edges in later chains since it was the head of $J_2$. 

\item $e_y$. This is a one-way chain with tail $v$ and head $y$. The tail $v$ has two (in fact three) incident edges in earlier chains ($e_x$, $e_z$, $e_w$). The head $y$ is either $r$ or incident to two edges in later chains since $y$ has degree one in $H_{CI(e_2)}$ and $y\notin V(J_2)$ by assumption, so $y$ has degree at least three in $\overline{H_{CI(e_2)}}$ unless it is $r$.
\end{enumerate}

\end{itemize}

We consider the impact of these replacements on other chains in the graph. As before, most of the edges of the old chains $J_1$ and $J_2$ were inserted at the same chain indices $CI(e_1)$ and $CI(e_2)$ respectively, preventing any changes. The pinched edges $e_1$ and $e_2$ were deleted, but the ends $x$, $z$, $w$ each received new incident edges $e_x$, $e_z$, $e_w$ inserted at the same chain indices ($CI(e_1)$, $CI(e_2)$, and $CI(e_2)$ respectively). However, $e_y$ was inserted at a different chain index than the deleted edge $e_1$ since $e_1$ was at $CI(e_1)$ while $e_y$ is at $CI(e_2)$. By the claim assumptions, $y$ has degree one in $H_{CI(e_2)}$, so there are no chains containing $y$ between $CI(e_1)$ and $CI(e_2)$, and so no chains were affected by the change. Thus, we have maintained the chain decomposition. This proves Claim \ref{clm2}.
\end{proof}

We may now assume the following for the remaining cases:
\begin{itemize}
\item If $J_1$ is a one-way chain, then $y$ has degree at least two in $H_{CI(e_2)}$.
\item If $J_2$ is a one-way chain, then $z$ has degree at least two in $\overline{H_{CI(e_1)}}$.
\end{itemize}

We also make the following conditional definitions, which will aid in distinguishing the remaining cases:

\begin{itemize}
\item If $J_1$ is a one-way chain and $y$ is not in $H_{CI(e_1)}$, then define the minimal index $i$ such that $y\in V(G_i)$ and $CI(e_1)<i<CI(e_2)$. Since $i$ is minimal, $y$ has degree one in $H_i$ (incident only to the pinched edge $e_1$). From this and the fact that $G_i$ is a minimal chain, it follows that either $y$ is one of two distinct ends of the up chain $G_i$, or $y$ is the head of the one-way chain $G_i$ which is not a loop.
\item If $J_2$ is a one-way chain and $z$ is not in $\overline{H_{CI(e_2)}}$, then define the maximal index $j$ such that $z\in V(G_j)$ and $CI(e_1)<j<CI(e_2)$. Since $j$ is maximal, $z$ has degree one in $\overline{H_j}$ (incident only to the pinched edge $e_2$). From this and the fact that $G_j$ is a minimal chain, it follows that either $z$ is one of two distinct ends of the down chain $G_j$, or $z$ is the tail of the one-way chain $G_j$ which is not a loop.
\end{itemize}


\begin{claim}\label{clm3}
Suppose that either one of $i,j$ is not defined, or $i<j$. Then $G'$ has a chain decomposition rooted at $r$.
\end{claim}

\begin{proof}
The chains replacing $J_1$ will have indices adjacent to $CI(e_1)$ and $i$ (if it is defined). Likewise, the chains replacing $J_2$ will have indices adjacent to $CI(e_2)$ and $j$ (if it is defined). Thus, by the assumptions of this claim, the chains replacing $J_1$ will have lower chain index than the chains replacing $J_2$. This fact will be needed when confirming that the new chains are valid. We begin by replacing $J_1$ as follows:

\begin{itemize}
\item Case 1: $J_1$ is an up chain. We replace it with $P_xe_xe_yP_y$. This is an up chain. The new vertex $v$ has two incident edges in later chains, namely $e_z$ and $e_w$.

\item Case 2: $J_1$ is a down chain. We replace it with the following chains (in the listed order):

\begin{enumerate}
\item $P_x$. Only add this chain if $P_x$ contains an edge. This is a down chain. The new end $x$ has an incident edge in a later chain, namely $e_x$.

\item $P_y$. Only add this chain if $P_y$ contains an edge. This is a down chain. The new end $y$ has an incident edge in a later chain, namely $e_y$.

\item $e_x$. This is a one-way chain with tail $x$ and head $v$. The tail $x$ is either $r$ or incident to two edges in earlier chains since it was in the old down chain $J_1$. The head $v$ has two incident edges in later chains, namely $e_z$ and $e_w$.

\item $e_y$. This is a one-way chain with tail $y$ and head $v$. The tail $y$ is either $r$ or incident to two edges in earlier chains since it was in the old down chain $J_1$. The head $v$ has two incident edges in later chains, namely $e_z$ and $e_w$.
\end{enumerate}

\item Case 3: $J_1$ is a one-way chain whose head $y$ is in $H_{CI(e_1)}$. We replace it with the following chains (in the listed order):

\begin{enumerate}
\item $e_x$. This is a one-way chain with tail $x$ and head $v$. The tail $x$ was the tail of the old one-way chain $J_1$. The head $v$ has two (in fact three) incident edges in later chains, namely $e_y, e_z, e_w$.

\item $e_y$. This is an up chain. The vertex $y$ is either $r$ or incident to two edges in later chains since it was the head of the old one-way chain $J_1$, and it has an incident edge in an earlier chain by assumption. The vertex $v$ has two incident edges in later chains, namely $e_z$ and $e_w$, and is incident to $e_x$ from the previous chain.
\end{enumerate}

\item Case 4: $J_1$ is a one-way chain whose head $y$ is not in $H_{CI(e_1)}$. Then $i$ is defined as above.

First, we replace $J_1$ with $e_x$. This is a one-way chain with tail $x$ and head $v$. The tail $x$ was the tail of the old one-way chain $J_1$. The head $v$ has two (in fact three) incident edges in later chains, namely $e_y, e_z, e_w$.

\begin{itemize}
\item Subcase 1: $y$ is one of two distinct ends of the up chain $G_i$. Replace $G_i$ with $G_ie_y$. This is an up chain. Since $G_i$ was a path and $v$ is a new vertex, this new chain is a path. The new end $v$ is adjacent to one edge in an earlier chain ($e_x$) and two edges in later chains ($e_z$ and $e_w$).

\item Subcase 2: $y$ is the head of the one-way chain $G_i$ which is not a loop. Then $y$ is not required to be in $H_i$ for $G_i$ to be a valid chain. In fact, $y$ is not required to be in any of $H_0,H_1,\ldots,H_i$ by the definition of $i$ and the assumptions of this case. Thus, we can leave $G_i$ as is and insert the chain $e_y$ immediately after $G_i$. This is an up chain. The vertex $y$ is incident to an edge in the previous chain $G_i$, and is either $r$ or incident to two edges in later chains since it is the head of $G_i$. The vertex $v$ is adjacent to one edge in an earlier chain ($e_x$) and two edges in later chains ($e_z$ and $e_w$).
\end{itemize}

\end{itemize}

The procedure for replacing $J_2$ is symmetric, by following the above steps in the reversed chain decomposition.

We consider the impact of these replacements on other chains in the graph. In most cases, we replaced the old chain $J_1$ with new chains inserted at the same chain index $CI(e_1)$, preventing any changes. The pinched edge $e_1$ was deleted, but the end $x$ received a new incident edge $e_x$ at the same chain index $CI(e_1)$. In Cases 1-3, the same is true for $y$. In Case 4, $y$ received a new incident edge $e_y$ either at or immediately after the chain index $i$. However, by the definition of $i$ and the claim assumptions, no chains were affected by the new chain index except $G_i$, which was specifically considered and shown to be valid in Case 4. By similar arguments, the changes caused by replacing $J_2$ also did not invalidate any chains. Thus, we have maintained the chain decomposition. This proves Claim \ref{clm3}.
\end{proof}


\begin{claim}\label{clm4}
Suppose that both of $i,j$ are defined and $i=j$. Then $G'$ has a chain decomposition rooted at $r$. 
\end{claim}

\begin{proof}
Recall that $G_i$ is either an up chain or a one-way chain with head $y$, and $G_j$ is either a down chain or a one-way chain with tail $z$. Since $i=j$, we conclude that $G_i=G_j$ must be a one-way chain with tail $z$ and head $y$, and $y\neq z$ since $i$ and $j$ are defined. We can replace $J_1$ and $J_2$ with the following chains, in the listed order. The first two will be placed immediately before index $i=j$, and the last two immediately after index $i=j$; see Figure~\ref{figClm4} for an illustration:

\begin{enumerate}
\item $e_x$. This is a one-way chain with tail $x$ and head $v$. The tail $x$ was the tail of the old one-way chain $J_1$ and we are placing this chain after index $CI(e_1)$. The head $v$ has two (in fact three) incident edges in later chains, namely $e_y, e_z, e_w$.

\item $e_z$. This is a one-way chain with tail $z$ and head $v$. By the definition of $j$, the tail $z$ is either $r$ or incident to two edges in earlier chains than $G_j$, and we are placing this chain immediately before index $j$. The head $v$ has two incident edges in later chains, namely $e_y$ and $e_w$.

\item $e_y$. This is a one-way chain with tail $v$ and head $y$. The tail $v$ has two incident edges in earlier chains, namely $e_x$ and $e_z$. By the definition of $i$, the head $y$ is either $r$ or incident to two edges in later chains than $G_i$, and we are placing this chain immediately after index $i$.

\item $e_w$. This is a one-way chain with tail $v$ and head $w$. The tail $v$ has two (in fact three) incident edges in earlier chains, namely $e_x, e_y, e_z$. The head $w$ was the head of the old one-way chain $J_2$, and we are placing this chain before $CI(e_2)$.
\end{enumerate}

\begin{figure}[h]
\centering
\includegraphics[scale=0.6]{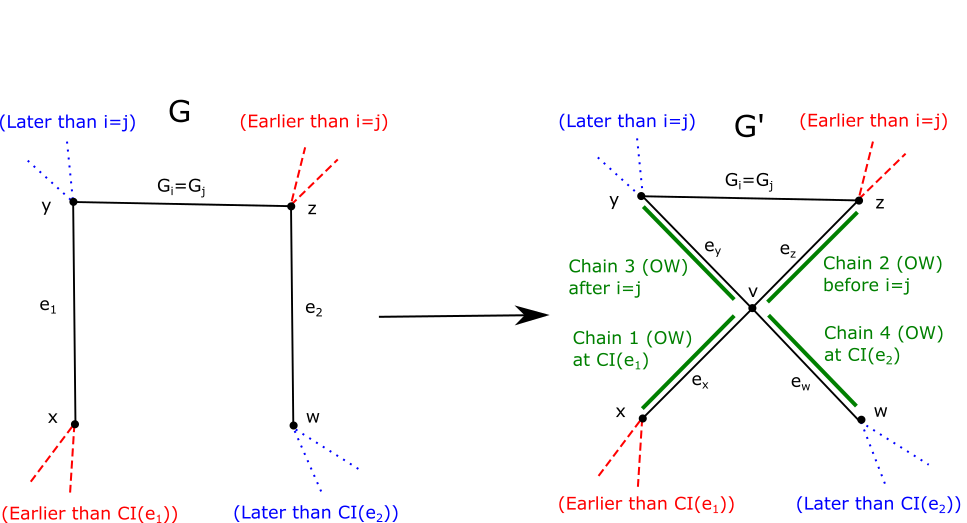}
\caption{An illustration of the procedure in Claim 4. The original chains $J_1$ and $J_2$ are on the left, while their replacements in $G'$ are on the right. The red/dashed edges are in earlier chains, while the blue/dotted edges are in later chains, with the particular meanings of ``earlier" and ``later" in the corresponding labels.}
\label{figClm4}
\end{figure}

We consider the impact of these replacements on other chains in the graph. The deleted edge $e_1$ was replaced by two edges with chain index greater than $CI(e_1)$, so we must be careful. The edge $e_x$ was inserted before index $i$, but $x$ had degree at least two in $H_{CI(e_1)}$, so losing a degree in later $H$ subgraphs will not invalidate any chains. The edge $e_y$ was inserted immediately after index $i$, so by the definition of $i$, the only chain affected is $G_i$. Since $G_i$ has $y$ as a head, losing a degree in $H_i$ will not invalidate the chain. By a symmetric argument, the changes caused by $e_z$ and $e_w$ do not invalidate any chains.
This proves Claim \ref{clm4}.
\end{proof}


\begin{claim}\label{clm5}
Suppose that both of $i,j$ are defined, and $i>j$. Then $G'$ has a chain decomposition rooted at $r$.
\end{claim}

\begin{proof}
We can replace $J_1$ and $J_2$ with the following chains, at the indicated chain indices; see Figure~\ref{figClm5} for an illustration:

\begin{enumerate}
\item $e_x$. Add this chain at index $CI(e_1)$. This is a one-way chain with tail $x$ and head $v$. The tail $x$ was the tail of the old one-way chain $J_1$ and we are placing this chain at index $CI(e_1)$. The head $v$ has two (in fact three) incident edges in later chains, namely $e_y, e_z, e_w$.

\item $e_z$. Add this chain immediately after $G_j$. This is a one-way chain with tail $z$ and head $v$. By the definition of $j$, the tail $z$ is either $r$ or incident to two edges in earlier chains than $G_j$, and we are placing this chain after index $j$. The head $v$ has two incident edges in later chains, namely $e_y$ and $e_w$.

\item $e_y$. Add this chain immediately before $G_i$. This is a one-way chain with tail $v$ and head $y$. The tail $v$ has two incident edges in earlier chains, namely $e_x$ and $e_z$. By the definition of $i$, the head $y$ is either $r$ or incident to two edges in later chains than $G_i$, and we are placing this chain before index $i$.

\item $e_w$. Add this chain at index $CI(e_2)$. This is a one-way chain with tail $v$ and head $w$. The tail $v$ has two (in fact three) incident edges in earlier chains, namely $e_x, e_y, e_z$. The head $w$ was the head of the old one-way chain $J_2$, and we are placing this chain at index $CI(e_2)$.
\end{enumerate}

\begin{figure}[h]
\centering
\includegraphics[scale=0.6]{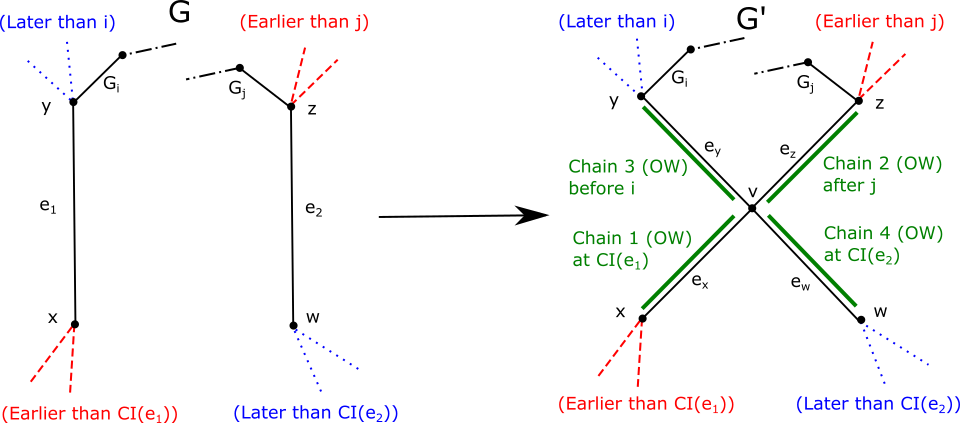}
\caption{An illustration of the procedure in Claim 5. The original chains $J_1$ and $J_2$ are on the left, while their replacements in $G'$ are on the right. The red/dashed edges are in earlier chains, while the blue/dotted edges are in later chains, with the particular meanings of ``earlier" and ``later" in the corresponding labels. The black/dashed-and-dotted segments represent paths which may have any length (including 0).}
\label{figClm5}
\end{figure}

We consider the impact of these replacements on other chains in the graph. The edge $e_1$ was deleted, but $x$ received a new incident edge $e_x$ at the same chain index $CI(e_1)$. The edge $e_y$ was inserted before index $i$, but the index is still smaller than $i$, so by the definition of $i$, no chains are affected. By a symmetric argument, the changes caused by $e_z$ and $e_w$ also do not invalidate any chains. This proves Claim \ref{clm5}.
\end{proof}

The claims cover all possibilities of pinching edges. The proof of Theorem~\ref{thmChains} is complete. The proof also implies a polynomial-time algorithm to construct a chain decomposition.\qed

\section{Proof of Theorem~\ref{thmTreesFromChains}}

Assume that we have a chain decomposition $G_0,G_1,\ldots,G_m$ of $G$. By Remark \ref{rmkMin}, we may assume that the chain decomposition is minimal. We will adapt the strategy of Curran, Lee, and Yu~\cite{cly} to prove Theorem~\ref{thmTreesFromChains}. In particular, we will construct two partial numberings of the edges of $G$ using the chain decomposition. We will then construct four spanning trees in two pairs, with one pair associated with each numbering. Within each pair, paths back to the root $r$ will be monotonic in the associated numbering to ensure independence. Between pairs, paths back to the root $r$ will be monotonic in chain index to ensure independence.

Using Corollary \ref{corNoLoops}, we may assume that there are no loops in $G$. By Lemma \ref{lmaDeg2}, for each vertex $v\neq r$, there are two distinct non-loop edges incident to $v$ whose chain indices are strictly smaller than the chain index of any other edge incident to $v$. Likewise there are two distinct edges whose chain indices are strictly larger than the chain index of any other edge adjacent to $v$. We will name these edges as follows:

\begin{definition*}
For each vertex $v\neq r$, the two \textit{$f$-edges} of $v$ are the two incident edges with the lowest chain index. Similarly, the two \textit{$g$-edges} of $v$ are the two incident edges with the highest chain index.
\end{definition*}

\begin{remark}\label{rmkfDown}
By the definition of a down chain, the edges of down chains are never $f$-edges. Likewise, by the definition of an up chain, the edges of up chains are never $g$-edges.
\end{remark}

Next, we will iteratively define a numbering $f$, which will assign distinct values in $\mathbb{R}$ to all edges in up chains and one-way chains. Here, two ``consecutive'' edges in a chain will refer to two edges in the chain which are incident to an internal vertex of the chain, so the two edges incident to the end of a closed chain are not consecutive, despite being adjacent.

We begin by numbering the edges in $E(G_0)$, and then number the edges of each up chain and one-way chain in order of chain index. When we reach a chain $G_i$, we may assume that all edges in $E(H_i)$ belonging to up chains and one-way chains have been numbered, which includes all $f$-edges in $E(H_i)$ by Remark \ref{rmkfDown}. We use the following procedure to number the edges in $E(G_i)$:

\begin{itemize}
\item If $G_i$ is a closed up chain containing $r$, then number the edges in $E(G_i)$ so that the values change monotonically between consecutive edges in the chain. The particular numbers used are arbitrary.

\item If $G_i$ is a closed up chain not containing $r$, then both $f$-edges of the common end have already been numbered. Call these two $f$-edges \textit{numbering edges} of $G_i$. Say the numbering edges of $G_i$ have $f$-values $a$ and $b$. Number the edges in $E(G_i)$ so that the values change monotonically between consecutive edges in the chain, and all values are between $a$ and $b$.

\item If $G_i$ is an open up chain containing $r$, then $r$ is an end and the other end is some $u\neq r$. At least one $f$-edge of $u$ has already been numbered. Choose an $f$-edge which has already been numbered and call it a \textit{numbering edge} of $G_i$. Say that $a$ is the $f$-value of the numbering edge. Number the edges in $E(G_i)$ so that the values increase between consecutive edges in the chain when moving from $u$ to $r$, and all values are larger than $a$.

\item If $G_i$ is an open up chain not containing $r$, then at least one $f$-edge of each end has been numbered. If the ends are $u$ and $v$, we can choose two distinct edges $e_u, e_v\in E(H_i)$ so that $e_u$ is an $f$-edge of $u$ and $e_v$ is an $f$-edge of $v$. We can choose these two distinct edges because otherwise, the only $f$-edge of $u$ or $v$ in $E(H_i)$ would be a single edge between $u$ and $v$, and then $H_i$ would not be connected. Call the edges $e_u,e_v$ \textit{numbering edges} of $G_i$. Without loss of generality, $f(e_u)=a<b=f(e_v)$. Number the edges in $E(G_i)$ so that the values increase between consecutive edges in the chain when moving from $u$ to $v$, and all values are between $a$ and $b$.

\item If $G_i$ is a one-way chain whose tail is $r$, then number the edge of $G_i$ arbitrarily.

\item If $G_i$ is a one-way chain whose tail is not $r$, then both $f$-edges of the tail are already numbered, say with $f$-values $a$ and $b$. Number the edge of $G_i$ between $a$ and $b$.
\end{itemize}

We symmetrically define a numbering $g$, which assigns distinct values in $\mathbb{R}$ to the edges of down chains and one-way chains, by using the above procedure in the reversed chain decomposition.

We are finally ready to construct the trees. Define the subgraphs $T_1,T_2,T_3,T_4$ as follows. For each $v\neq r$, consider the two $f$-edges of $v$. Assign the edge with the lower $f$-value to $T_1$ and the edge with the higher $f$-value to $T_2$. Similarly, consider the two $g$-edges of $v$. Assign the edge with the lower $g$-value to $T_3$ and the edge with the higher $g$-value to $T_4$.

Several properties of $T_1,T_2,T_3,T_4$ will follow from the following claim.

\begin{claim*}
For any $v\neq r$, consider the edge $e_1$ assigned to $T_1$ at $v$. Let $v'$ be the other end of $e_1$. If $v'\neq r$, let $e'_1$ be the edge assigned to $T_1$ at $v'$. Then $CI(e'_1)\leq CI(e_1)$ and $f(e'_1)<f(e_1)$.
\end{claim*}

\begin{proof}
Let $e_2$ be the edge assigned to $T_2$ at $v$. The edge $e_1$ is not in a down chain by Remark \ref{rmkfDown}. We break into two cases.

\begin{itemize}
\item Suppose $e_1$ is in an up chain $G_i$. Since the chain decomposition is minimal and $v'\in V(G_i)$, its $f$-edges are either in $E(G_i)$, or else have chain index less than $i$. In either case, $CI(e'_1)\leq i=CI(e_1)$ as desired.

Note that $e_2$ is either in $E(G_i)$, or else is the numbering edge of $G_i$ at the end $v$. By the numbering procedure, we know that $f(e_1)$ is between $f(e_2)$ and the $f$-value of one of the $f$-edges of $v'$, say $e^*$. By the definition of $T_1$, $f(e_1)<f(e_2)$, so it follows that $f(e^*)<f(e_1)$. Again by the definition of $T_1$, $f(e'_1)\leq f(e^*)$, so $f(e'_1)<f(e_1)$ as desired.

\item Suppose $e_1$ induces a one-way chain $G_i$. Since $e_1$ is an $f$-edge, $v$ has degree at most one in $H_i$, so $v$ must be the head of $G_i$. Then $v'$ is the tail of $G_i$, so the $f$-edges of $v'$ have chain indices smaller than $i$, which means $e'_1\neq e_1$ and $CI(e'_1)<CI(e_1)$ as desired.

From the numbering procedure, we know that $f(e_1)$ is between the $f$-values of the two $f$-edges of $v'$, with $f(e'_1)$ being the smaller by the definition of $T_1$. So, $f(e'_1)<f(e_1)$ as desired.
\end{itemize}

In both cases we have $CI(e'_1)\leq CI(e_1)$ and $f(e'_1)<f(e_1)$. This proves the claim.
\end{proof}

With the claim proven, it follows that the edges assigned to $T_1$ are all distinct, there are no cycles in $T_1$, and following consecutive edges assigned to $T_1$ produces a path which is decreasing in chain index, strictly decreasing in $f$-value, and can only end at $r$. Thus, $T_1$ is connected and is a spanning tree of $G$. A similar argument shows that $T_2$ is a spanning tree of $G$ where paths to $r$ are decreasing in chain index and strictly increasing in $f$-value. Due to the opposite trends in $f$-values, $T_1$ and $T_2$ are edge-independent with root $r$.

By symmetry, we obtain analogous results for $T_3$ and $T_4$. It remains to show that a tree from $\{T_1,T_2\}$ and a tree from $\{T_3,T_4\}$ are edge-independent. The paths back to $r$ from a vertex $v\neq r$ are decreasing in chain index in one tree and increasing in chain index in the other tree, but not strictly. The first edges in these paths are an $f$-edge and a $g$-edge of $v$, respectively. By Lemmas \ref{lmaDeg2} and \ref{lmaMinDeg4}, there is a positive difference in chain index between these initial edges, so the paths are in fact edge-disjoint. The proof of Theorem \ref{thmTreesFromChains} is complete. The proof also implies a polynomial-time algorithm to construct the edge-independent spanning trees.\qed

\section{Summary of Results}

With Theorems \ref{thmChains} and \ref{thmTreesFromChains} proven, we obtain Theorem \ref{thmTrees}. In fact, we can examine the argument more carefully to extract a stronger, summarizing result.

\begin{corollary}\label{corSummary}
Suppose $G$ is a graph with no isolated vertices and $V(G)\geq2$. Then the following statements are equivalent.
\begin{enumerate}
\item $G$ is $4$-edge-connected.
\item There exists $r\in V(G)$ so that $G$ has a chain decomposition rooted at $r$.
\item For all $r\in V(G)$, $G$ has a chain decomposition rooted at $r$.
\item There exists $r\in V(G)$ so that $G$ has four edge-independent spanning trees rooted at $r$.
\item For all $r\in V(G)$, $G$ has four edge-independent spanning trees rooted at $r$.
\end{enumerate}
\end{corollary}

\begin{proof}
Theorem \ref{thmChains} gives us $(1)\Rightarrow(3)$. Theorem \ref{thmTreesFromChains} gives us $(2)\Rightarrow(4)$ and $(3)\Rightarrow(5)$. Trivially, we have $(3)\Rightarrow(2)$ and $(5)\Rightarrow(4)$. Therefore, we need only show $(4)\Rightarrow(1)$.

Assume for the sake of contradiction that $G$ has four edge-independent spanning trees rooted at some $r\in V(G)$, but is not $4$-edge-connected. Suppose $S\subseteq E(G)$ is an edge cut with $\abs{S}<4$. Consider a vertex $v$ in a component of $G-S$ not containing $r$. Using the paths in each of the edge-independent spanning trees, we find that there exist four edge-disjoint paths between $v$ and $r$. This contradicts the existence of $S$.
\end{proof}

\vfill

This material is based upon work supported by the National Science Foundation. Any
opinions, findings, and conclusions or recommendations expressed in this material are those
of the authors and do not necessarily reflect the views of the National Science Foundation.

\end{document}